\documentclass[11pt,twoside,a4paper,reqno]{amsart}
\usepackage[english]{babel}
\usepackage{amssymb,graphicx,color}
\usepackage{mathrsfs}

\usepackage[pdfauthor={Yu. Brezhnev}, colorlinks, citecolor=blue,
            pdfwindowui=false,
            hyperfootnotes=false,
            bookmarksopen=false,bookmarks=false]{hyperref}

\usepackage{CutPage}
\textcenter{150mm}{202mm} \croppage{1mm}{0mm}

\usepackage{YURA}
\usepackage{mySymbols,myScripts,myTOF}
\usepackage[blue]{myFoots}

\theoremstyle{plain}
\newtheorem{theorem}{Theorem}

\newtheorem{corollary}[theorem]{Corollary}
\newtheorem{proposition}[theorem]{Proposition}

\theoremstyle{definition}

\theoremstyle{remark}
\newtheorem{remark}{Remark}

\def\u{{\mathfrak{u}}}

\def\pp{\hbox{\protect\smaller[2]$\mathcal{P}$}}
\def\q{\hbox{\protect\smaller[2]$\mathcal{Q}$}}
\def\R{{\mathscr{R}}}
\def\RT{{\mathfrak{T}}}
\def\G{{\bo{\mathfrak{G}}}}
\def\HG{\mathbb{H}^+\!/\G}
\def\GR{{\G_{\!\!\s[5]\R}}}
\def\bcQ{{\bo{\mathcal{Q}}}}
\def\ellK{\mathsf{K}}
\def\btau{{\bo{\tau}}}
\def\balpha{{\bo{\alpha}}}
\def\bbeta{{\bo{\beta}}}
\def\bgamma{{\bo{\gamma}}}
\def\bdelta{{\bo{\delta}}}
\def\dvphi{\,\dot{\!\vphi}}
\def\tsint{\raise-0.15ex\hbox{\Larger[2]{\smallint}}\!}

\title[Analytic connections on Riemann surfaces]
{Analytic connections\\on Riemann surfaces and orbifolds}

\author[Yu.~Brezhnev]{Yurii V.~Brezhnev}

\thanks{Research was supported by the Tomsk State University Academic
D.~Mendeleev Fund Program.}


\address{Department of Quantum Field Theory, Tomsk State University,
634050, Tomsk, Russia}

\email{brezhnev@mail.ru}

\date{\today}
\date{10th March 2015} 

\begin{document}

\hfill{\smaller[3]Journal of Geometry and Physics (2015) {\bf 97},
166--176}

\hfill{\blue\smaller[3]\ttt{http://dx.doi.org/10.1016/j.geomphys.%
2015.07.005}}\\\\\\

\begin{abstract}
We give a differentially closed description of the uniformizing
representation to the analytical apparatus on Riemann surfaces and
orbifolds of finite analytic type. Apart from well-known automorphic
functions and Abelian differentials it involves construction of the
connection objects. Like functions and differentials, the
connection, being also the fundamental object, is described by
algorithmically derivable \odes. Automorphic properties of all of
the objects are associated to different discrete groups, among which
are excessive ones. We show, in an example of the hyperelliptic
curves, how can the connection be explicitly constructed. We study
also a relation between classical/traditional `linearly
differential' viewpoint (principal Fuchsian equation) and
uniformizing $\tau$-representation of the theory. The latter is
shown to be supplemented with the second (to the principal) Fuchsian
equation.
\end{abstract}

\maketitle

\tableofcontents \thispagestyle{empty}

\newpage
\section{Introduction}

\noindent Riemann surfaces are of fundamental importance to the
mathematical physics because most effective part of the modern
differential/integral calculus is related, one way or the other,
with a complex analysis on a certain Riemann surface of (or not) a
certain analytic function. Surfaces of a finite genus are
distinctive in that the calculus is the best elaborated one with lot
of applications. Strangely enough, the standard differential
apparatus on such kind objects of higher genera ($g>1$) cannot be
considered as completely closed; this remark requires some
explanation.

Let $\R$ be a finite genus Riemann surface  determined by
irreducible algebraic equation
\begin{equation}\label{1}
F(x,y)=0\,.
\end{equation}
Uniformizing representation of $\R$ is given by a pair of
single-valued analytic functions $x=\vphi(\tau)$, $y=\vpsi(\tau)$,
wherein the global uniformizer $\tau$ belongs to the upper
half-plane $\Hp$, that is $\Im(\tau)>0$. Functions $\vphi$ and
$\vpsi$ are the automorphic ones with respect to an infinite
discrete Fuchsian group $\GR$ \cite{ford}:
\begin{equation}\label{scalar}
\vphi\mbig[7](\Mfrac{\alpha\,\tau+\beta}{\gamma\,\tau+\delta}
\mbig[7])=\vphi(\tau)\,,\qquad
\vpsi\mbig[7](\Mfrac{\alpha\,\tau+\beta}{\gamma\,\tau+\delta}
\mbig[7])=\vpsi(\tau)\,,\qquad \forall\tau\in\Hp\,,
\end{equation}
where $\big(\begin{smallmatrix}\alpha&\beta\\
\gamma&\delta\end{smallmatrix}\big)\in\GR
\subset\mathrm{PSL}_2(\mathbb{R})$. Complex analysis on $\R$
includes Abelian differentials $R(x,y)\,dx$, their integrals $\tsint
R(x,y)\,dx$, and the $\R$ itself is completely determined by periods
of Abelian integrals that are holomorphic (everywhere finite)
\cite{farkas}. We know also that if some function $\psi(\tau)$ is a
$\tau$-representation for any of the differentials above then its
automorphic property is characterized by a weight-2 automorphic form
\cite{ford,forster,farkas}:
\begin{equation}\label{diff}
\psi\mbig[7](\Mfrac{\alpha\,\tau+\beta}{\gamma\,\tau+\delta}\mbig[7])
=(\gamma\,\tau+\delta)^2\cdot\psi(\tau)\,.
\end{equation}

In the uniformization theory automorphic functions are described by
the 2nd order linear ordinary differential equations (\odes) of a
Fuchsian class \cite{ford}
\begin{equation}\label{Q}
\Psi_{\!\!\mathit{xx}}=\frac12\,\bcQ(x,y)\,\Psi\,,
\end{equation}
where $\bcQ$ is a rational function of its arguments. In general,
this is a Fuchsian equation with algebraic coefficients. The
uniformizing parameter $\tau$ is then defined as a ratio
\begin{equation}\label{tau}
\tau=\frac{\Psi_2(x)}{\Psi_1(x)}
\end{equation}
of the linearly independent solutions to Eq.~\eqref{Q} and inversion
of this ratio determines one of the uniformizing functions:
$x=\vphi(\tau)$. It is well known that the theory is equivalent to
the 3rd order \ode\ $\{\tau,x\}=-\bcQ(x,y)$ containing no the
auxiliary $\Psi$-function, where $\{\tau,x\}$ is the standard
notation for the Schwarz derivative \cite{ford}
\begin{equation}\label{SD}
\{\tau,x\}\DEF\frac{\tau_{\mathit{xxx}}}{\tau_x}-\frac32\,
\frac{\pow{\tau}{\mathit{xx}}{2}}{\pow{\tau}{\mathit{x}}{2}}\,.
\end{equation}

Since the theory is described by the third order \odes, its complete
data set is not exhausted by functions \eqref{scalar} and their
first order differentials \eqref{diff}: the second order
differentiation is missing. Alternatively, the $\R$ may be thought
of as a 1-dimensional complex manifold \cite{forster} and all the
objects above can be treated from the differential geometric
viewpoint. Then functions \eqref{scalar}--\eqref{diff} represent
scalars and 1-forms and calculus should involve a covariant
differentiation of these and other tensor fields. By this means, in
order to close the complex analytic theory, we have to introduce (at
least) a canonical bundle over our $\R$ and corresponding connection
object $\Gamma$. Partially, some ingredients of such a view on the
theory have already been appeared in the literature. Dubrovin
\cite{dubrovin} gave a geometric treatment  to the famous Chazy
equation $\pi\,\dddot{\smash[b]{\eta}}=
12\,\ri\,(2\,\eta\,\ddot\eta-3\,\dot\eta^2)$ when the group $\GR$ is
the genus zero full modular group
$\mathrm{PSL}_2(\mathbb{Z})\FED\bo{\Gamma}(1)$  and Hawley \&
Schiffer \cite{hawley} introduced the connection $\Gamma$ in the
context of conformal mappings of planar domains and multi-connected
representations of $\R$. The well-known modular forms \cite{zagier}
are the particular cases of automorphic forms when group is a
subgroup of $\bo{\Gamma}(1)$. They possess interesting differential
properties and some of them---\odes\ for some low level groups
$\Gamma_0(N)$---are constructed in \cite{maier}. It may  be remarked
here that even the theory of the $\bo\Gamma(1)$-connection function,
\ie, the Chazy--Weierstrass function $\eta(\tau)$, is not restricted
by the Chazy equation mentioned above. Recent work \cite{jon}
provides an alternative theory (and nontrivial application) in the
language of linear Fuchsian equations \eqref{Q}.

The known examples \cite{dubrovin,maier,zagier} are concerned only
with the zero genus cases and general automorphic properties of
bundles and connections on them, to our knowledge, are not
considered in the literature. This is the subject matter of the
present work. We give an analytically closed geometric description
for the differential calculus on Riemann surfaces of finite analytic
type (genus and number of punctures are finite) through the
uniformizing $\tau$-representation for the connection objects
$\Gamma(\tau)$ and characterize their differential properties. More
precisely, not only do functions and differentials satisfy some
autonomic 3rd order \odes\ (the known fact
\cite{ford,hurwitz,zagier}), but connections also satisfy equations
of such a kind. What is more, a remarkable property of (analytic)
connections on $\R$'s of arbitrary genera is the fact that all of
them come from a trivial connection on certain orbifolds of the
\emph{zero} genus and satisfy autonomic \odes. These \odes\ are
algorithmically derivable.

\section{Invariant quantities on $\R$}

\subsection{Invariant counterparts of Fuchsian equations}

Let Fuchsian equation \eqref{Q} determine, through its monodromical
group $\G_x$, an exact representation of fundamental group
$\pi_1^{}$ of a certain orbifold or the $\R$ itself. However, from
differential geometric viewpoint this equation is not well defined;
it has no invariant (autonomic) form. Indeed, it contains explicitly
the quantity $x$ which, in a generic case $F_y(x,y)\ne 0$, is the
standard usage for a local coordinate on $\R$. In turn the quantity
$\tau(x)$ coming from the definition \eqref{tau} obviously does not
produce the geometric object. However  we may take 1-dimensionality
of $\R$ into account and swap around the standard coordinate $x$ and
`object' $\tau$; thus $x$ may be thought of as the scalar `field'
quantity $x=x(\pp)$ \cite{nevanlinna} being represented by function
$x=\vphi(\tau)$  on the universal cover $\Hp$. Then the automorphic
property \eqref{scalar} becomes nothing but the $\HG$-factor
topology reformulation to the property of $x$ to be a scalar:
\begin{equation}\label{xp}
\tilde x(\pp)=x(\pp)\,,\qquad
\vphi\mbig[7](\Mfrac{\alpha\,\tau+\beta}{\gamma\,\tau+\delta}
\mbig[7])=\vphi(\tau)\,;
\end{equation}
here $\tilde x(\pp)$ is a value of the quantity $x$ at point
$\pp\in\R$ under the  coordinate choice\footnote{Greek symbols
$\alpha$, $\beta$, \ldots\ will be used for discrete group
transformations and Latin $a$, $b$, \ldots\ for coordinate changes.}
\begin{equation}\label{proj}
\tilde \tau=\frac{a\,\tau+b}{c\,\tau+d}
\end{equation}
and arbitrariness of the  real numbers $(a,b,c,d)$ comes from
well-known projective structure on $\Hp$. The second generator
$y(\pp)=\vpsi(\tau)$ of the field of meromorphic functions on $\R$
is determined by the same properties as \eqref{xp}. Therefore we
should do an inverting the Schwarzian \eqref{SD} into the object
$\{x,\tau\}$ and, denoting $[x,\tau]\DEF-\{\tau,x\}$, introduce an
invariant (coordinate-free) form of the principal equation
\eqref{Q}:
\begin{equation}\label{main}
[x,\tau]=\bcQ(x,y)\,,
\end{equation}
where
\begin{equation}\label{MD}
[x,\tau]\DEF
\frac{\dddot{\smash[b]{x}}}{\dot x^3}-\frac32\,\frac{\ddot x^2} {\dot
x^4}
\end{equation}
and the dot above a symbol stands for a $\tau$-derivative. Proof
uses the known property of the Schwarz derivative
$$
-\dot x^2\,\{\tau,x\}=\{x,\tau\}\,.
$$

Thus the 3rd order differential object $[x,\tau]$ represents, due to
equation \eqref{main}, a scalar function on $\R$. It follows also
that the nonlinear differential operator $z\mapsto [z,\tau]$
generates scalar objects from the scalar ones. Indeed, the
well-known transformation rule for the Schwarz derivative of a
function composition $\tau=f(\mu)$ and $\mu=g(z)$ \cite{ford}, \ie,
\begin{equation}\label{tensor}
\{\tau,\mu\}\,d\mu^2+\{\mu,z\}\,dz^2=\{\tau,z\}\,dz^2\,,
\end{equation}
being rewritten in terms of the $[x,\tau]$-objects, implies that if
$z$ is any rational (scalar) function on $\R$, that is $z=R(x,y)$,
then
$$
[z,\tau]=[R,x]+\frac{1}{\pow{R}{x}{2}}\,\bcQ(x,y)\,;
$$
this expression is again the rational function on $\R$.

As for  the Abelian differentials, their two structure
properties---analogs of \eqref{xp}---have obviously the following
form \cite{forster}:
\begin{equation*}\label{Diff}
\widetilde{\psi}(\pp)=\frac{d\tau}{d\tilde\tau}\cdot\psi(\pp)\,,
\qquad\psi\mbig[7](\Mfrac{\alpha\,\tau+\beta}{\gamma\,\tau+\delta}
\mbig[7])=(\gamma\,\tau+\delta)^2\cdot\psi(\tau)\,,
\end{equation*}
where $\widetilde{\psi}(\pp)$ represents the 1-form
$\omega=\psi(\pp)\,d\tau$ in the coordinate $\tilde\tau$. Behavior
of the higher order analytic differential $k$-forms
$f(\pp)\,d\tau^k$ is defined in a similar manner:
$$
\skew4\widetilde{f}(\pp)=\mbig[7](\frac{d\tau}{d\tilde\tau}\mbig[7])
^{\!\!k}\cdot f(\pp)\,,
\qquad
f\mbig[7](\Mfrac{\alpha\,\tau+\beta}{\gamma\,\tau+\delta}\mbig[7])
=(\gamma\,\tau+\delta)^{2k}\cdot f(\tau)\,,
$$
Since the 1-dimensional complex analysis under construction deals
with the analytic differentials and their powers, we have in fact to
supplement the base $\R$ with a canonical (cotangent, holomorphic,
and line) bundle $K$ \cite{accola,jost}.

\subsection{Automorphic property for a connection}

The transformation properties of geometric objects in  cotangent
bundles are completely determined by transformations in the base
$\R$. Therefore we may take the known transformation law for a
connection form $\Gamma_{\![4]\alpha}$ and adopt it to our
1-dimensional case $T^*(\R)$; clearly, $\Gamma_{\![4]\alpha}$ must
also respect the projective structure \eqref{proj}. For an arbitrary
vector bundle with a structure group $G$ we have \cite{novikov}
$$
\widetilde\Gamma_{\![4]\alpha}=\frac{\partial z^\beta}{\partial
\tilde z^\alpha}\, \Big(G\,\Gamma_{\![4]\beta}\, G^{\sm1}-
\frac{\partial
G}{\partial z^\beta}\,G^{\sm1} \Big)\,,
$$
where the simultaneous transformations of coordinates in the base
$z\mapsto \tilde z$ and in a fiber $\Psi\mapsto \widetilde\Psi$ are
carried out:
$$
z\mapsto \tilde z=\tilde z(z)\,,\qquad \Psi\mapsto \widetilde\Psi=
G(\pp)\,\Psi\,.
$$
Of course, the covariant differentiation $\nabla$ is defined here by
the standard rule:
$\nabla_{\!\!\alpha}\Psi=\frac{\partial\,\Psi}{\partial z^\alpha}-
\Gamma_{\!\!\alpha}\Psi$. Let us change notation for the old/new
coordinates $(z,\tilde z)\to(\tau, \tilde\tau)$ and take into
account that for cotangent bundles we have to put
$$
G(\pp)=\frac{d\tau}{d\tilde \tau}\,.
$$
Applying all this to the case under consideration, we get
\begin{equation}\label{GP}
\widetilde\Gamma(\pp)=\frac{d\tau}{d\tilde\tau}\!\mbig[6]|_{\!\pp}
\cdot\Gamma(\pp)+
\mbig[7](\frac{d}{d\tilde\tau}\!
\ln\frac{d\tau}{d\tilde\tau}\mbig[7])_{\!\!\pp}
\end{equation}
and the above mentioned projective structure leads to the following
property.

\begin{proposition}\label{P1}
The transformation rule for an analytic connection on $T^*(\R)$
realized on the universal cover $\Hp$ reads as
\begin{equation}\label{tmp}
(a\,d-b\,c)\,\widetilde\Gamma\!\mbig[7](\Mfrac{a\,\tau+b}{c\,\tau+d}
\mbig[7])=(c\,\tau+d)^2\cdot\Gamma(\tau)+2\,c\,(c\,\tau+d)\,,
\end{equation}
where $\{a,b,c,d\}$ are real.
\end{proposition}

In order to pass to the $\tau$-representation of our $\R$'s we now
need to satisfy the factorization of the $\Hp$-topology with respect
to some discrete group $\G$ acting on $\Hp$. To put it differently,
we have to obtain a property of the function $\Gamma(\tau)$
representing on $\Hp\!\!/\G$ the connection object $\Gamma(\pp)$.

\begin{theorem}\label{T1}
Let $\big(\begin{smallmatrix}\alpha&\beta\\
\gamma&\delta\end{smallmatrix}\big)\in\GR$ be an exact matrix
representation of $\pi_1^{}(\R)$ and $\Gamma(\tau)$ be a
uniformizing representation of the connection object on $T^*(\R)$.
Then
\begin{equation}\label{Gauto}
(\alpha\,\delta-\beta\,\gamma)\,\Gamma\!
\mbig[7](\Mfrac{\alpha\,\tau+\beta}
{\gamma\,\tau+\delta}\mbig[7])
=(\gamma\,\tau+\delta)^2\cdot\Gamma(\tau)+
2\,\gamma\,(\gamma\,\tau+\delta)\,.
\end{equation}
\end{theorem}
\begin{proof}
Rewrite the property \eqref{GP} in form of `separated' differentials:
$$
\widetilde\Gamma(\pp)\,d\tilde\tau|_{{\mathcal{P}}}^{\mathstrut}=
\Gamma(\pp)\,d\tau|_{{\mathcal{P}}}^{\mathstrut}+(d\ln\!d\tau)_
{{\mathcal{P}}}^{\mathstrut}-
(d\ln\! d\tilde\tau)_{{\mathcal{P}}}^{\mathstrut}\,.
$$
Hence
$$
\Gamma(\pp)\,d\tau|_{{\mathcal{P}}}^{\mathstrut}+(d\ln\! d\tau)_
{{\mathcal{P}}}^{\mathstrut}=
\widetilde\Gamma(\pp)\,d\tilde\tau|_{{\mathcal{P}}}^{\mathstrut}
+
(d\!\ln d\tilde\tau)_{{\mathcal{P}}}^{\mathstrut}
$$
and the quantity
$\Gamma(\pp)\,d\tau|_{{\mathcal{P}}}^{\mathstrut}+(d\ln\! d\tau)_
{{\mathcal{P}}}^{\mathstrut}$ is thus a scalar invariant. Let
points $\pp$ and $\q$ be equivalent  with respect to group $\GR$,
that is $\pp\sim \q$. We then may equate  values of the latter
scalar invariant taken at $\pp$ and $\q$:
$$
\Gamma(\pp)\,d\tau|_{{\mathcal{P}}}^{\mathstrut}+(d\ln\! d\tau)_
{{\mathcal{P}}}^{\mathstrut}=
\Gamma(\q)\,d\tau|_{{\mathcal{Q}}}^{\mathstrut}+(d\ln\! d\tau)_
{{\mathcal{Q}}}^{\mathstrut}\,.
$$
Pass to the notation $\tau\DEF\tau|_{{\mathcal{P}}}^{\mathstrut}$
and $\btau\DEF\tau|_{{\mathcal{Q}}}^{\mathstrut}$; one gets
$$
\Gamma(\q)=\frac{d\tau}{d\btau}\cdot\Gamma(\pp)+
\frac{d}{d\btau}\!\ln\frac{d\tau}{d\btau}\,.
$$
Substituting here the group transformations
$\btau=\frac{\alpha\,\tau+\beta}{\gamma\,\tau+\delta}$, we arrive at
formula \eqref{Gauto}.
\end{proof}

\begin{remark}\label{R1}
Relations \eqref{tmp} and \eqref{Gauto} are the properties in their
own rights. The first one is local and defines the object
$\Gamma(\pp)$: the two functions $\Gamma$, $\widetilde\Gamma$ are
evaluated at one point $\pp$. The second property is nonlocal: one
function representing the object itself  is evaluated at two
$\tau$-points. Obviously, the covariant differentiation of the
$k$-differentials above is defined as follows \cite{dubrovin}
$$
\nabla f(\tau)=\frac{d}{d\tau}f(\tau)-k\,\Gamma(\tau)\,f(\tau)\,.
$$
In the case of modular subgroups of the group $\bo{\Gamma}(1)$ the
functions possessing the formal property \eqref{Gauto} are sometimes
called the quasi-modular forms \cite{maier,zagier}.
\end{remark}

\section{Construction of connections}

\subsection{Zero genus orbifolds}

Let us consider first the case of genus zero. It is clear, that
nontrivial theory appears only if the zero genus sphere
$\R=\mathbb{C}\mathrm{P}^1$ is endowed with points wherein this $\R$
ceases to be a manifold with a trivial fundamental group and each of
these points determines a finite order element of the fundamental
group $\pi_1^{}$ (the conical singularity) or a group element of
infinite order (the puncture). We thus have to consider an
$N$-punctured sphere, which is an orbifold $\RT$ with the generic
group $\pi_1^{}(\RT)=\langle\mathfrak{a}_1^{},\ldots,
\mathfrak{a}_{N\sm1}^{}\rangle$, where $(\mathfrak{a}_s)^{p_s^{}}=1$
and integral $p$'s are formally allowed to be equal to $\infty$. In
this case the theory is described by a Fuchsian equation with
rational coefficients of the following form:
\begin{equation}\label{Qx}
\begin{aligned}
\Psi_{\!\!\mathit{xx}}&=\frac12\,\bcQ(x)\,\Psi\\
&=\frac14
\mbig[8]\{\frac{p_s^{\sm2}-1}{(x-e_s)^2}+\cdots\mbig[8]\}\!\Psi\,,
\end{aligned}
\end{equation}
where $x\in\RT=\overline{\mathbb{C}}\backslash\{e_1^{},\ldots,e_{{
N}\sm1}^{},\infty\}$ and accessory parameters (hidden in dots) have
been chosen such that the monodromy representation of
$\pi_1^{}(\RT)$ be a first kind Fuchsian group \cite{ford}.

Since this base $\RT$ is noncompact, every holomorphic vector bundle
on it is trivial \cite{forster}.  On the other hand, the map
$x=\vphi(\tau)$ between $\RT$ and its fundamental polygon for
$\pi_1^{}(\RT)$ is  single-valued in both directions; $\vphi$ is a
Hauptmodul. Hence we may take $x$ as the global coordinate on $\RT$,
think of it as a flat one  $\tilde\tau\DEF x$, and consider in this
(old) coordinate the zero connection:
$\widetilde\Gamma(\pp)\equiv0$. For the new (non-flat) coordinate
$\tau$ we therefore have, according to the law \eqref{GP},
$$
0=\frac{1}{\dot x}\,\Gamma(\pp)-\frac{d}{dx}\!\ln \dot x
\qquad\hence\qquad
\Gamma(\tau)=\frac{d}{d\tau}\!\ln\dvphi(\tau)\,;
$$
whence it follows that such a $\Gamma$ comes from a scalar function
on $\RT$. Since the difference of any two connections
$\Gamma-\Gamma'$ on $T^*(\RT)$ is a differential, we obtain the
following property.

\begin{proposition}\label{P2}
Uniformizing representations to the analytic everywhere holomorphic
connections on $T^*(\RT)$ for a zero genus orbifold $\RT$ are
determined by its Hauptmodul $x=\vphi(\tau)$\/$:$
\begin{equation}\label{Gx}
\Gamma(\tau)=\frac{d}{d\tau}\!\ln\dvphi(\tau)+R\big(\vphi(\tau)\big)
\dvphi(\tau)\,,
\end{equation}
where $R$ is a holomorphic function on $\RT$.
\end{proposition}

Because we are interested in the \textit{meromorphic} analysis on
Riemann surfaces and orbifolds we handle only with meromorphic, \ie,
Abelian differentials and, therefore, put the function $R$ to be a
rational one on $\overline{\mathbb{C}}$ with poles at points
$\{e_k^{},\infty\}$ at most. It defines a holomorphic differential
$R(x)\,dx$ on $\RT$. If we allow for $R(x)$ to have poles on $\RT$
then one can speak of meromorphic connections.

Another kind arguments for construction of the connection above uses
the fact that for a \mbox{1-di}\-men\-sional case the curvature of
an analytic connection is an identical zero and we may look for a
covariantly constant section $\psi$ (differential), that is
$$
\psi_\tau^{}=\Gamma(\tau)\,\psi\,,
$$
where $\Gamma(\tau)$ is as yet unknown. Differentials do certainly
exist and all of them are generated from one of them, say, $dx$ by
the formula above $R(x)\,dx$. Substituting here
$\psi=R(x)\,\dvphi(\tau)$ we arrive again at a formula of the form
\eqref{Gx}.

\subsection{Relation between arbitrary and zero genera}

The transition from the previous case of $g=0$ to the arbitrary
genera is based on the following extension of a result implicitly
formulated in \cite{whittaker}.

\begin{theorem}\label{P3}
For a compact Riemann surface defined by an arbitrary algebraic
curve \eqref{1} there exists a function field  $\mathbb{C}(x,y)$
generator pair $(z,w)$ such that one of the generators, \eg, $z$ has
a zero genus automorphism group $\mathrm{Aut}(z(\tau))\FED\G_z$ and
is determined by a Fuchsian equation with rational
coefficients\/\footnote{To avoid lengthening terminology we call the
linear equations \eqref{Q} and their nonlinear counterparts
\eqref{main} Fuchsian.}
\begin{equation}\label{Q38}
[z,\tau]=-\frac38
\mbig[9]\{ \sum_{s=1}^{2n+1}\frac{1}{(z-E_s)^2}-
\frac{2\,n\,z^{2n-1}+A(z)}{(z-E_1)\cdots(z-E_{2n+1})}\mbig[9]\},
\end{equation}
where $A(z)$ is a properly  chosen accessory polynomial of degree
$2\,n-2$.
\end{theorem}
\begin{proof}
By an appropriate birational substitution
$(x,y)\rightleftarrows(z,w)$ one can always transform the curve
\eqref{1} into a nonsingular form $\widetilde F(z,w)=0$ having only
the simple branch points $z=E_s$. This means that multi-valued
function $w=w(z)$  has a local ramification structure of the form
$(w-w_s)^2=a_s\,(z-E_s)+\cdots$ for all $E$'s with $a_s\ne 0$ and
holomorphic otherwise: $w-w_0^{}=b_0^{}(z-z_0^{})+\cdots$ under
$z_0\ne E_s$. The number of $E$-points is always even (denote it as
$2\,n+2$) and we can put one of them at $z=\infty$. Consider
Fuchsian equation \eqref{Q38}. The Klein--Poincar\'e theorem
\cite{ford} states that there is a unique $A(z)$-polynomial such
that this equation determines the globally single-valued on $\Hp$
analytic function $z=z(\tau)$. The coefficient $-\frac38$ in
\eqref{Q38} says that $z(\tau)$ has the following local behavior in
neighborhoods of $E$'s: $z=E+\tau^2+\cdots$. At infinity we have the
development $z=\tau^{\sm2}+\cdots$. Therefore function
$w(\tau)=w(z(\tau))$ is everywhere single-valued as well.
Eq.~\eqref{Q38} has no other singularities except $\{E_s\}$ and,
hence, in neighborhoods of the regular points
$(z_0^{},w_0^{})\in\widetilde F$ we have the developments
$$
z=z_0^{}+A\,(\tau-\tau_0^{})+\cdots\,,\qquad
w=w_0^{}+B\,(\tau-\tau_0^{})+\cdots\,,
$$
which are holomorphic. Since the pair $(z,w)$ has been obtained with
the help of birational transformation, the functions $z=R_1(x,y)$
and $w=R_2(x,y)$ form just a different pair of generators of the
function field: $\mathbb{C}(x,y)=\mathbb{C}(z,w)$. We thus have a
complete conformal purely hyperbolic $\Hp$-image $\big(z(\tau),
w(\tau)\big)$ of \eqref{1}: $\R=\Hp\!/\GR$.
\end{proof}

From the group-algebraic viewpoint this theorem means that the exact
matrix representation of $\pi_1^{}(\R)=\GR$ for a compact Riemann
surface defined by \eqref{1} may be represented as an intersection
of different monodromy group pairs
\begin{equation}\label{GR}
\GR=\G_x\mbox{\large$\cap$}\,\G_y=
\G_z\mbox{\large$\cap$}\,\G_w=\cdots
\end{equation}
and one of the generators here (say, $z$) has a zero genus monodromy
$\G_z$; this is not obvious a priori. To put it differently, when
describing purely hyperbolic higher genera Riemann surfaces  the
\emph{zero genus} orbifolds do always appear and their structure
properties are determined by \emph{ `rational'} Fuchsian equations.
(Recall that zero genus monodromy group may correspond both to a
rational and to an algebraic Fuchsian equation). Hence, the function
elements $z$ with a zero genus automorphism $\mathrm{Aut}(z(\tau))$
do certainly exist and there are no reasons to ignore such objects
when constructing the effective theory of the $\R$ itself. A simple
explanation here is the fact that wider groups are easier described
and this point should certainly be exploited in the theory.

Equation \eqref{Q38} determines the function $z(\tau)$ which,
besides being a defining Hauptmodul for a zero genus orbifold
$\RT_z$, is a function element $z=R_1(x,y)$ of the `nonzero genus
field' $\mathbb{C}(x,y)$ of rational functions on \eqref{1}. Hence
it follows that whatever compact $g>1$ Riemann surface $\R$ may be
there always exist associated zero genus orbifolds whose Hauptmoduln
are elements of the function field $\mathbb{C}(x,y)$. By this means
we can construct connections on nonzero genera $\R$'s using
connections on $T^*(\RT_z)$. The only thing we need is to satisfy
the automorphic property \eqref{Gauto} for the full group $\GR$.

\subsection{Excessive automorphisms}

Let us take a connection on $T^*(\RT_z)$ defined by formula
\eqref{Gx}
$$
\Gamma(\tau)=\frac{d}{d\tau}\!\ln\dot z(\tau)
$$
and add to it an arbitrary Abelian differential on the curve
\eqref{1}:
\begin{equation}\label{Gz}
\Gamma(\tau)=\frac{d}{d\tau}\!\ln\dot z(\tau)+R(z,w)\,\dot z(\tau)\,.
\end{equation}
In this way we obtain the desired connection on $T^*(\R)$ because
automorphic property \eqref{Gauto} for the object \eqref{Gz} does
certainly hold for the group $\GR$ if $R(z,w)$ is not a function of
$z$ alone; all the transformations from group
$\G_z\mbox{\large$\cap$}\,\G_w$ fit into the formula \eqref{Gauto}.
This condition on function $R(z,w)$ is of course necessary  but not
sufficient because the  rule \eqref{Gauto} may take place for a
group that may be wider than $\GR$; just as group $\G_z$ may be
wider than group $\GR$.

Such an `excessive' extension may occur when constructing
differentials and even functions: different differentials on one
orbifold may possess the automorphic property \eqref{diff} with
respect to different groups forming towers of subgroups. For
example, Whittaker's family of curves $w^2=z^{2g+1}+1$ (the
classical Weierstrass case $g=1$ is not an exception) provides
counterexamples when both the generators $z$, $w$ and the base
differentials $dz$ and $dw$ have automorphic properties
\eqref{scalar}--\eqref{diff} with respect to groups that are larger
than $\GR$. In all these cases genera of $\G_z$ and $\G_w$ are equal
to zero and groups $\G_w$ are even the simple triangle ones:
$$
[z,\tau]=-\frac38\,
\frac{z^{2g+1}-4\,g(g+1)}{(z^{2g+1}+1)^2}\,z^{2g-1}\,,\qquad
[w,\tau]=-\frac{2\,g\,(g+1)}{(2\,g+1)^2}\,
\frac{w^2 + 3}{(w^2-1)^2}\,.
$$
These equations are solvable in terms of hypergeometric
${}_2F_1$-functions \cite{whittaker3,dalzell}. Concerning functions,
one view on this problem is a tessellation of a Fuchsian
$\R$-polygon on the schlicht (univalent) domains of a given
$u(\tau)$-function. Surprisingly, in such a formulation the problem
is not elaborated even in the elliptic $g=1$ case. In general we
should involve, according to \eqref{GR}, automorphic objects with
both the generating groups $\G_z$ and $\G_w$. Hence the problem
above can be reduced to a problem of searching for an automorphic
function $u(\tau)=R\mbig[1](z(\tau),w(\tau)\mbig[1])$ whose
automorphism group $\G_u$ coincides with  group
$\GR=\G_z\mbox{\large$\cap$}\,\G_w$. Complexity of the problem is
related to the theory of Fuchsian equations with algebraic
coefficients; the latter has almost not been developed.

\begin{remark}
To avoid lengthening terminology we shall apply the notion
automorphism $\G$ not only to functions (scalars) but to
differentials and connections as well in the sense that the rules
\eqref{diff} and \eqref{Gauto} respect the group $\G$. This
automorphism group is assumed, by definition, to be a \emph{maximal}
set of group elements under which the corresponding transformation
laws hold.
\end{remark}

\begin{proposition}\label{C1}
Uniformizing representations for meromorphic differentials on $\R$
and connections on $T^*(\R)$  lift to differentials and connections
for manifolds $\skew3\widetilde{\R}$ that cover
$($finitely-sheeted\/$)$ $\R$.
\end{proposition}
\begin{proof}
Fundamental group of a covering manifold is a subgroup of the one
being covered \cite{novikov}.
\end{proof}

Thus, since every group can be embedded in a larger group, the
problem consists in elimination of the objects (functions,
differentials, and connections) having wider automorphisms than
exact representations of $\pi_1^{}(\R)$. Fortunately, there is an
algorithmic solution in the case of hyperelliptic curves.

\subsection{Hyperelliptic case}

The aim of this section is to derive a class of algebraic
dependencies \eqref{1} for which the connection object
$\Gamma(\tau)$ is explicitly constructed and whose automorphism
coincides with group $\GR$. By formula \eqref{Gz} any connection is
built by a logarithmic derivative of a differential and we begin
with differentials, namely, non-exact differentials because
functions (exact differentials) may have wider automorphisms than
$\GR$.

Let $d\,\mathfrak{A}=R(x,y)\,d\u$, where $\u$ is a holomorphic
differential, be an Abelian differential whose $\tau$-representation
$\dot{\mathfrak{A}}(\tau)$ (formula \eqref{diff})  respects the
group $\G_{\mathfrak{A}}\supset\GR$ with finite index
$|\G_{\mathfrak{A}}:\GR|\ne1$. In other words, assume that in
addition to transformations
$\big(\begin{smallmatrix}\alpha&\beta\\
\gamma&\delta\end{smallmatrix}\big)\in\GR$ there exists a
transformation
$\big(\begin{smallmatrix}\balpha&\bbeta\\
\bgamma&\bdelta\end{smallmatrix}\big)\notin\GR$ such that
\begin{equation}\label{A}
(\balpha\,\bdelta-\bbeta\,\bgamma)
\,R\mbig[7](\!\vphi\Big(\Mfrac{\balpha\,\tau+\bbeta}
{\bgamma\,\tau+\bdelta}\Big),
\vpsi\Big(\Mfrac{\balpha\,\tau+\bbeta}
{\bgamma\,\tau+\bdelta}\Big)\!\mbig[7])
\dot\u\Big(\Mfrac{\balpha\,\tau+\bbeta}
{\bgamma\,\tau+\bdelta}\Big)=
(\bgamma\,\tau+\bdelta)^2\,
R\mbig[1](\vphi(\tau),\vpsi(\tau)\mbig[1])\,\dot\u(\tau)\,.
\end{equation}
Let us analyze location of zeroes/poles of the differential because
they are the well-defined and generate infinities of $\Gamma$. The
latter in turn are the only well-defined objects for connections.
Denote by $\tau_0^{}$ coordinate of one of the zeroes
$\pp_{\!\!0}^{}$: $\dot{\mathfrak{A}}(\tau_0^{})=0$. This zero
(right hand side of \eqref{A}) may get mapped into a zero of the
$\big(\begin{smallmatrix}
\balpha&\bbeta\\
\bgamma&\bdelta\end{smallmatrix}\big)$-transformed holomorphic part
$\dot\u$, of the meromorphic one $R$, or into both of them (left
hand side of \eqref{A}). Hence `rational' part of $d\,\mathfrak{A}$
can take up zeroes of the holomorphic one $\dot\u$ and generate  new
zeroes. In order to get around such `intertwining and canceling' the
zeroes/poles of $\dot\u$ and $R$ we drop out the rational part
$R(x,y)$ and consider only the holomorphic differential:
$$
(\balpha\,\bdelta-\bbeta\,\bgamma)
\,\dot\u\mbig[7](\Mfrac{\balpha\,\tau+\bbeta}
{\bgamma\,\tau+\bdelta}\mbig[7])=
(\bgamma\,\tau+\bdelta)^2\,
\dot\u(\tau)\,.
$$
It follows that $\tau'=\frac{\balpha\,\tau_0^{}+\bbeta}
{\bgamma\,\tau_0^{}+\bdelta}$ also represents a zero:
$\dot\u(\pp')=0$. Clearly, $\pp'$ may not be a zero being
\mbox{$\GR$-equivalent} of $\pp_{\!\!{\s0}}$; for we should
otherwise have
$\mbig[3](\begin{smallmatrix}\balpha&\bbeta\\
\bgamma&\bdelta\end{smallmatrix}\mbig[3])=
\mbig[3](\begin{smallmatrix}\alpha&\beta\\
\gamma&\delta\end{smallmatrix}\mbig[3])$. Therefore $\tau'$ may only
be an image of another zero and analysis is complicated if $\dot\u$
has several zeroes. On the other hand, if we take $\dot\u$ having a
single zero, its position should be preserved by the
$\mbig[3](\begin{smallmatrix}\balpha&\bbeta\\
\bgamma&\bdelta\end{smallmatrix}\mbig[3])$-transformation and it
must be an elliptic one. The existence of differentials with
elliptic automorphisms on purely hyperbolic $\GR$ follows from
Theorem~\ref{P3}; whence it follows that the rational differential
$R(z)\,\dot z$ has even the order two automorphisms at all the
points $E$'s. It has however excessive zeroes so we drop out its
rational part $R$ and, in order to reduce the zeroes, divide $\dot
z$ by a function having simple zeroes at $E$'s, \ie, by
$\prod_{E}\!\sqrt{z-E}$. One kind of algebraic curves for which this
is possible suggests itself: this is the hyperelliptic class
\begin{equation}\label{hyper}
w^2=(z-E_1)\cdots(z-E_{2g+1})\,,\qquad \FED P(z)
\end{equation}
with parametrization $z=\vphi(\tau)$, $w=\vpsi(\tau)$. We put in
this case
\begin{equation*}\label{infty}
\dot\u=\frac{\dvphi}{\vpsi}
\end{equation*}
and this differential, being holomorphic, has the only zero at
$\pp_{\!\!\s0}=(\infty,\infty)$. Renormalizing the $\tau$-plane, we
can assign $\tau_0^{}=0$ to this point and impart the form
$\tau\mapsto \varepsilon\tau$ to the elliptic transformation, that
is
$\mbig[3](\begin{smallmatrix}\balpha&\bbeta\\
\bgamma&\bdelta\end{smallmatrix}\mbig[3])=
\mbig[3](\begin{smallmatrix}\varepsilon&0\\
0&1\end{smallmatrix}\mbig[3])$,  where $\varepsilon^n=1$. The case
under consideration  corresponds to  $n=2$ and hence $z=\vphi(\tau)$
is an even function and $w=\vpsi(\tau)$ is an odd one:
$$
\vphi(-\tau)=\vphi(\tau)\,,\qquad
\vpsi(-\tau)=-\vpsi(\tau)\,.
$$
This is nothing but the hyperelliptic involution  $(z,w)\mapsto
(z,-w)$ (see also \cite{whittaker}). Now, let us build the following
connection:
\begin{equation}\label{Gu}
\Gamma(\tau)=\frac{d}{d\tau}\!\ln\dot\u(\tau)=
\frac{d}{d\tau}\!\ln\frac{\dvphi}{\vpsi}\,.
\end{equation}
Compact Riemann surfaces are however purely hyperbolic manifolds but
this connection admits our elliptic transformation in the sense that
\eqref{Gu} transforms like differential $\dvphi$:
$\dvphi(-\tau)=-\dvphi(\tau)$. Indeed,
$$
\Gamma(\tau)=\frac{\,\ddot{\!\vphi}}{\dvphi}-
\frac{\,\,\dot{\!\!\vpsi}}{\vpsi}=
\frac{\,\ddot{\!\vphi}}{\dvphi}-\frac12\,P'(\vphi)
\frac{\dvphi}{\vphi}
$$
and this $\Gamma(\tau)$ has the same parity as $\dvphi$. Moreover,
any connection built by arbitrary hyperelliptic differential
respects a wider automorphism:
$$
\Gamma=\frac{d}{d\tau}\!\ln\!R(z)\,\frac{\dot z}{w}=
\frac{d}{d\tau}\!\ln\!R(z)\,\dot z-
\frac12\frac{d}{d\tau}\!\ln w^2=
\frac{d}{d\tau}\!\ln\!R(z)\,\dot z-\frac12\,\sum_E(z-E)^{\sm1}
\cdot\dot z\,,
$$
which is  actually a connection on orbifold $\Hp\!/\G_z$. Thus, in
order to avoid  excessive elliptic automorphisms we should add to
connection \eqref{Gu} any pole-free (nonzero) differential
insensitive to permutation of sheets. Clearly, this is done by a
holomorphic differential:
\begin{equation}\label{final}
\Gamma(\tau)=\frac{d}{d\tau}\!\ln\frac{\dvphi}{\vpsi}+
(c_1^{}+c_2^{}\vphi+\cdots +c_g^{}\vphi^{g\sm1})\,
\frac{\dvphi}{\vpsi}\,.
\end{equation}
We thus have arrived at the following result.

\begin{theorem}
Let $\GR$ be a matrix representation of $\pi_1^{}(\R)$ for a
hyperelliptic Riemann surface \eqref{hyper}. Then connection
$\Gamma(\tau)$ defined by \eqref{final}, subjected to a condition
that at least one of $c_j^{}\ne0$, has an automorphism $\GR$ and
represents a connection object on $T^*(\R)$ with a single pole at
infinite point $\pp_{\!\s0}=(\infty,\infty)$.
\end{theorem}

\begin{proof}
The transformation law \eqref{Gauto}, where $\GR=\G_z\cap\G_w$, is
satisfied by construction. Suppose that \eqref{final} respects a
`larger' factor topology on certain $\Hp\!\!/\G$:
\begin{equation}\label{Gp}
(\balpha\,\bdelta-\bbeta\,\bgamma)\,
\Gamma\!\mbig[7](\Mfrac{\balpha\,\tau+\bbeta}
{\bgamma\,\tau+\bdelta}\mbig[7])
=(\bgamma\,\tau+\bdelta)^2\cdot\Gamma(\tau)+
2\,\bgamma\,(\bgamma\,\tau+\bdelta)\,,\qquad
\mbig[3](\begin{smallmatrix}\balpha&\bbeta\\
\bgamma&\bdelta\end{smallmatrix}\mbig[3])
\in \G \supset \GR\,.
\end{equation}
Owing to uniqueness of singularities of $\Gamma(\tau)$, this
transformation must be an elliptic one of the 2nd order and we have
instead of \eqref{Gp}
$$
-\Gamma(-\tau)=\Gamma(\tau)\,.
$$
However this rule is not satisfied by expression \eqref{final}
because its two terms have opposite parities: the first one is odd
and the first one is even, a contradiction; the transformation
$\mbig[3](\begin{smallmatrix}\balpha&\bbeta\\
\bgamma&\bdelta\end{smallmatrix}\mbig[3])$ may only be an identical
one.
\end{proof}

\begin{remark}
Modifying the holomorphic differential
$$
\dot\u=(z-E_k)^{g\sm1}\,\frac{\dot z}{w}\,,
$$
we get connections with single singularities at points
$\pp_k=(E_k,0)$. Connections with single singularities (elementary,
in terminology of \cite{hawley}) may serve as building blocks for
construction of all the connections: by adding to them  meromorphic
differentials, one obtains connections with singularities at
prescribed points. Moreover all the $\Gamma$'s have an invariant
related to their poles. Indeed, integrating formula \eqref{GP}, we
derive
$$
\int_{\partial\R}\!\!\!\!\widetilde{\Gamma} (\tilde\tau)\,d\tilde\tau
=\int_{\partial\R}\!\!\!\!\Gamma(\tau)\,d\tau+\int_{\partial\R}
\!\!\!\!d\ln\frac{d\tau}{d\tilde\tau}=
\int_{\partial\R}\!\!\!\!\Gamma(\tau)\,d\tau
\qquad(\text{invariant})\,,
$$
since $\frac{d\tau}{d\tilde\tau}$ is everywhere holomorphic without
zeroes. Putting here $\Gamma=\frac{d}{d\tau}\!\ln\dot u(\tau)$ with
arbitrary differential $\dot u$, one can compute the value of this
invariant if $\R$ is a compact surface (not orbifold). In this case
difference between number of zeroes and poles of any differential
$\dot u$ is equal to $2\,g-2$ and we obtain that
$$
\int_{\partial\R}\!\!\!\!\Gamma(\tau)\,d\tau=(2\,g-2)\cdot
2\,\pi\ri\,.
$$
\end{remark}

\section{Differential properties of connections}

\subsection{ODEs for connections\label{odes}}

This is a classical result by Hurwitz that the automorphic weight-2
forms satisfy the third order \odes\ \cite{hurwitz,zagier}. These
forms are the standard objects in the complex analysis on $\R$'s and
in this section we shall show that the similar result takes place
for any connection object.

Let us consider the finite genus hyperbolic Riemann surface (or
orbifold) $\R$. Apart from algebraic equation \eqref{1} this object
is described by the Fuchsian equation for a meromorphic function
$u=R(x,y)$ on $\R$ whose uniformizing form $u(\tau)$ has an
automorphism $\G_u$ coinciding with $\pi_1^{}(\R)$. Corresponding
equation \eqref{main} does certainly exist because $\pi_1^{}$ for
such an object can always be realized as a monodromy of the 2nd
order linear Fuchsian \ode. We have
\begin{equation*}\label{Qu}
[u,\tau]=\bcQ(x,y)\,.
\end{equation*}
Construct first a connection $\Gamma(\tau)$ on $T^*(\R)$ with the
help of an exact differential $\dot u$:
$$
\Gamma=\frac{d}{d\tau}\!\ln\dot u
$$
and define the corresponding covariant differentiation of
$k$-differentials $\nabla=\partial_\tau-k\,\Gamma(\tau)$. Clearly,
$\nabla u=\dot u$ and $\nabla^2 u\equiv0$. Making use of the
property
$$
[u,\tau]\,\dot u^2=\frac{d}{d\tau}\!
\mbig[7]( \frac{d}{d\tau}\!\ln\dot u\mbig[7])-
\frac12\mbig[7](\frac{d}{d\tau}\!\ln\dot u\mbig[7])^{\!\!2},
$$
we observe that the object\footnote{The object $\ellK$ is an analog
of the standard curvature. Since $u$ may serve as a local coordinate
on $\R$, this connection $\Gamma$ and its `curvature' $\ellK$ may be
thought of as having almost everywhere zero values (in the flat
coordinate $u$) or as having the $\delta$-function like
distributions concentrated at a finite number of points wherein
$\dot u(\tau)=\{0,\infty\}$.}
\begin{equation}\label{K}
\ellK\DEF\dot{\mathrm{\Gamma}}-\frac12\,\Gamma^2
\end{equation}
is a 2-differential: $\ellK=\bcQ(x,y)\,(\nabla u)^2$. Taking the
$\nabla$-derivative, we have
$$
\nabla\ellK=\nabla\bcQ(x,y)\cdot(\nabla u)^2\,.
$$
The quantities $x(\tau)$ and $y(\tau)$ represent the scalar objects
on $\R$ (not only on $\Hp\!/\G_x$ or $\Hp\!/\G_y$;
Corollary~\ref{C1}), that is $\dot x=\nabla x$ and $\dot y=\nabla
y$. Hence we derive
$$
\nabla\ellK=(\bcQ_x\,\dot x+\bcQ_y\,\dot y)\cdot(\nabla u)^2=
\pow{F}{y}{\sm1}( F_y\bcQ_x-F_x\bcQ_y)\,\dot x\,\cdot(\nabla u)^2\,,
$$
where $\nabla\ellK\DEF\dot\ellK-2\,\Gamma\ellK$. The definition
$u=R(x,y)$ gives
$$
\dot x=\frac{F_y}{F_yR_x-F_xR_y}\,\nabla u
$$
and, consequently,
$$
\ellK=\bcQ\cdot(\nabla u)^2\,,\qquad
\nabla\ellK=\frac{F_y\bcQ_x-F_x\bcQ_y}{F_yR_x-F_xR_y}\cdot
(\nabla u)^3\,.
$$
Elimination of $\nabla u$ gives the identity
\begin{equation*}\label{N1}
\frac{(\nabla\ellK)^2}{\ellK^3}=
\frac{(F_y\bcQ_x-F_x\bcQ_y)^2}{(F_yR_x-F_xR_y)^2}\,\bcQ^{\sm3}\,.
\end{equation*}
The second $\nabla$-derivative yields
$$
\nabla^2\ellK=\frac{F_y}{F_yR_x-F_xR_y}
\,\frac{d}{dx}
\frac{F_y\bcQ_x-F_x\bcQ_y}{F_yR_x-F_xR_y}\cdot (\nabla u)^4\,,
$$
where $\nabla^2\ellK\DEF\big(\frac{d}{d\tau}-3\,\Gamma\big)
(\dot\ellK-2\,\Gamma\ellK)$. As before, elimination of $\nabla u$
produces the second scalar identity
$$
\frac{\nabla^2\ellK}{\ellK^2}=
\frac{F_y\bcQ^{\sm2}}{F_yR_x-F_xR_y}\,\frac{d}{dx}
\frac{F_y\bcQ_x-F_x\bcQ_y}{F_yR_x-F_xR_y}\,.
$$
As a result we obtain  three identities
\begin{equation}\label{ST}
\frac{(\nabla\ellK)^2}{\ellK^3}=S(x,y)\,,\qquad
\frac{\nabla^2\ellK}{\ellK^2}=T(x,y)\,,\qquad F(x,y)=0
\end{equation}
with certain rational functions $S$ and $T$. Eliminating here the
variable $x$ followed by $y$, one derives the equation
$L(\ellK,\nabla\ellK,\nabla^2\ellK)=0$, where $\ellK$ is understood
to be expressed as \eqref{K} and $\nabla\ellK$ and $\nabla^2\ellK$
as above. This is nothing but the 3rd order  \ode\ satisfied by the
connection object $\Gamma$. By construction \odes\ \eqref{ST} and
their polynomial consequences are invariant. We thus have arrived at
the following result.

\begin{theorem}\label{T2}
Let $\R$ be an arbitrary compact Riemann surface defined by equation
\eqref{1} or an arbitrary orbifold whose compactification is
equivalent to \eqref{1}. Let $\Gamma(\tau)$ be a uniformizing
representation for a connection object $\Gamma$ on $T^*(\R)$. Then
$\Gamma(\tau)$ satisfies the  algorithmically derivable 3rd order
autonomic polynomial \ode
\begin{equation}\label{Xi}
\Xi(\dddot{\smash[b]{\,\Gamma}},\ddot{\mathrm{\Gamma}},
\dot{\mathrm{\Gamma}},\Gamma)=0
\end{equation}
whose general solution is given by the following single-valued
analytic function\/$:$
$$
\Gamma=\frac{a\,d-b\,c}{(c\,\tau+d)^2}\,
\Gamma\Big(\Mfrac{a\,\tau+b}{c\,\tau+d}\Big)-
2\,c\,\frac{(a\,d-b\,c)}{c\,\tau+d}
$$
with free constants $\{a,b,c,d\}$.
\end{theorem}
\begin{proof}
The last formula follows from the transformation law \eqref{tmp} and
the only thing is left to be proved is that any other connection
satisfies the certain \ode\ \eqref{Xi}.

Let $\tilde\Gamma$ be such a connection:
$\tilde\Gamma=\Gamma-r(x,y)\,\dot u$. Redefine the objects used
above:
$$
\tilde\nabla=\frac{d}{d\tau}-k\,\tilde\Gamma\,,\qquad
\tilde\ellK\DEF{\tilde\Gamma}_{\!\tau}-\frac12\,\tilde\Gamma
$$
and take into account that $u$, in contrast to the preceding, is no
longer a flat coordinate:
$$
\tilde\nabla u=\dot u\,,\qquad \tilde\nabla^2 u=
r(x,y)\,(\tilde\nabla u)^2\,.
$$
It follows that
\begin{equation}\label{22}
\tilde\ellK=\frac{d}{d\tau}(\Gamma-r\,\dot u)-
\frac12\,(\Gamma-r\,\dot u)^2=
\dot{\mathrm{\Gamma}}-\frac12\,\Gamma^2-\dot r\,\dot u-
\frac12\,r^2\dot u^2\,.
\end{equation}
Let prime $'$, as always in the sequel, stand for the total
$x$-derivative:
$$
f'\DEF f_x-\frac{F_x}{F_y}\,f_y\,.
$$
Since $\dot u=R'\,\dot x$, expression \eqref{22} can be rewritten as
follows
$$
\tilde\ellK=\bigg(\bcQ-\frac{r'}{R'}-\frac12\,r^2 \bigg)
(\tilde\nabla u)^2\FED Q(x,y)\cdot(\tilde\nabla u)^2\,.
$$
Hence we compute
$$
\tilde\nabla\tilde\ellK=
\frac{1}{R'}\,(Q'+2\,r\,QR')\cdot(\tilde\nabla u)^3\,,\qquad
\tilde\nabla^2\tilde\ellK=
\frac{1}{R'}\!
\left\{\!\mbig[7](\frac{Q'}{R'}\mbig[7]{)}'+
2\,Q\,r'+5\,r\,Q'+6\,r^2QR'\right\}
\cdot(\tilde\nabla u)^4
$$
and therefore
$$
\frac{(\tilde\nabla\tilde\ellK)^2}{\tilde\ellK^3}=
\frac{(Q'+2\,r\,QR')^2}{Q^3R'^2}\,,\qquad
\frac{\tilde\nabla^2\tilde\ellK}{\tilde\ellK^2}=
\frac{1}{Q^2R'}\!
\left\{\!\mbig[7](\frac{Q'}{R'}
\mbig[7]{)}'+2\,Q\,r'+5\,r\,Q'+6\,r^2QR'\right\}.
$$
As before,  equation of the form \eqref{Xi} follows by elimination
of the pair $(x,y)$.
\end{proof}

\subsection{Differentials, connections, and differential closedness}

Since invariant \odes\ follow from differential properties of
automorphic functions, there should exist an equivalent
`non-invariant' description/explanation in the traditional language
of linear \odes.

Let us return to Eq.~\eqref{Q}. It is clear, that its two linearly
independent solutions
$$
\Psi_1(x)=\sqrt{\dot x\,}\,,\qquad \Psi_2(x)=\tau\,\sqrt{\dot x\,}
$$
are not differentially closed. Therefore from differential viewpoint
the \emph{complete} and \emph{closed} differential apparatus on $\R$
must involve not only the principal equation \eqref{Q}:
\begin{equation}\label{Psi}
\Psi_{\!\!\mathit{xx}}=\frac12\bcQ(x,y)\,\Psi\,,\tag{$\ref{Q}'$}
\end{equation}
but equation for a derivative $\Phi$ of the $\Psi$-function:
\begin{equation}\label{Phi}
\Phi\DEF\Psi_{\!x}\qquad\hence\qquad
\Phi_{\!\mathit{xx}}-\ln'\!\!\!\bcQ(x,y)\cdot\Phi_x-
\frac12\bcQ(x,y)\,\Phi=0\,.
\end{equation}
Clearly, this equation also belongs to a Fuchsian class and has its
proper monodromy group\footnote{It is an interesting problem to
study the  extended set of Fuchsian equations
\eqref{Psi}--\eqref{Phi} in the framework of uniformization theory.
This question, including a series of examples, will be the subject
matter of a separate work. When $\GR$ is the zero genus group
$\bo{\mathrm{\Gamma}}(1)$ a closure of the `$\Psi$-theory' into the
`$\Phi$-one' is considered in \cite{jon} in the context of the
modelling selforganised patterns of vegetation.}. Since connections
are defined up to differentials, we can set one of them as follows
$$
\Gamma(\tau)=\frac{d}{d\tau}\!\ln\!\dvphi(\tau)\qquad\hence\qquad
\Gamma(\tau)=2\,\Psi\,\Psi_{\!x}=2\,\frac{\dot{\mathrm{\Psi}}}{\Psi}
$$
and, therefore, introduction of a connection and the
$\Psi$-derivative are in fact the equivalent operations. Complete
set of data for the theory can thus be written in both the $x$- and
$\tau$-representation:
$$
\big\{ \Psi_1(x),\,\Psi_2(x),\,\Psi'\!\!\!{}_1(x)\big\}\quad
\scalebox{1.5}[1]{$\Leftrightarrow$}\quad\big\{\vphi(\tau),
\dvphi(\tau),\Gamma(\tau) \big\}\,.
$$
In this context the above mentioned 3rd order  \odes\ satisfied by
the 1- and 2-differential (say, $\ellK$) are also the direct and
algorithmical consequences of the linear equations. Indeed, first we
take the base differential $\psi=\dot x$, that is
$\psi=\pow{\Psi}{1}{2}$. Taking into account that
$\frac{d}{dx}=\psi^{\sm1}\frac{d}{d\tau}$, we obtain the following
$\tau$-equivalents of \eqref{Psi}--\eqref{Phi}:
\begin{equation}\label{df1}
2\,\ddot{\psi}\,\psi-3\,\dot\psi^2
=2\,\bcQ(x,y)\,\psi^4\,,\qquad
\dddot\psi\,\psi^2-6\,\ddot{\psi}\,\dot{\psi}\,\psi+6\,
\dot\psi^3=\bcQ'(x,y)\,\psi^6\,.
\end{equation}
Pass now to the arbitrary differential $\psi=R^{\sm1}(x,y)\,\dot x$.
Then these two  expressions should be changed according to the rules
$$
\psi\to R\cdot\psi\,,\qquad
\frac{d}{d\tau}\psi\to R\,R'\cdot\psi^2+
R\cdot\dot\psi\,,\qquad\frac{d}{d\tau}R\to R\,R'\cdot\psi\,.
$$
As in the case of connections (Sect.~\ref{odes}), we get  three
polynomial equations
\begin{equation}\label{df2}
S\big(\psi,\,\dot\psi,\,\ddot\psi;x,y\big)=0\,,
\qquad
T\big(\psi,\,\dot\psi,\,\ddot{\psi},\,
\dddot\psi;x,y\big)=0\,,\qquad F(x,y)=0
\end{equation}
which are, by construction, are independent of coordinate choice
\eqref{proj}. By elimination of $(x,y)$ the differential
$\psi=\psi(\tau)$ satisfies a certain autonomic \ode\ of the form
$\tilde\Xi\big(\psi,\,\dot\psi,\,\ddot\psi,\, \dddot\psi\big)=0$
whose general solution is
$$
\psi=\frac{a\,d-b\,c}{(c\,\tau+d)^2}\,
\psi\Big(\Mfrac{a\tau+b}{c\tau+d}\Big)\,,
$$
where $\psi(\tau)$ is any particular one.

Thus all the analytic geometric objects on arbitrary orbifolds of
finite genus are constructed to be governed by invariant (autonomic)
\odes: scalars are described by Eq.~\eqref{Q}, connections by
Eq.~\eqref{Xi}, and 1-differentials by Eqs.~\eqref{df1} or
\eqref{df2}.

It is worthy of special emphasis that we have nowhere used the fact
that the group $\G_x$ must be of 1st  Fuchsian kind  acting on $\Hp$
(with unique accessory parameters in \eqref{Psi}). All the
statements above hold for $\bcQ$-functions with different values of
accessory parameters determining the 2nd kind Fuchsian (Kleinian)
groups acting in $\overline{\mathbb{C}}$ without invariant circle
(groups of Schottky, Weber, and Burnside \cite{ford}).

\begin{corollary}\label{C2}
Let $\R$ be a hyperbolic Riemann surface or an orbifold of finite
analytic type with uniformizing group $\GR$ of Fuchsian or Schottky
type. Then the $\tau$-representations for meromorphic differentials
and connection objects on $T^*(\R)$ satisfy the algorithmically
derivable autonomic {\smaller[2]ODE}s of 3rd order.
\end{corollary}

Algorithmical constructions here are the same as in the case of
$\Hp$ described above, except that the different values of accessory
parameters in linear equations \eqref{Q} will lead to different
nonlinear autonomic equations for differentials and connections. It
may be noted here that these \odes\ carry  all the information about
group and, if group is $\GR$, about Riemann surface itself. They can
serve as an alternative to the linear but non-autonomic Fuchsian
\odes\ and deserve to be further investigated in their own rights.

\end{document}